\documentclass[twoside,11pt]{amsart}

\usepackage{a4wide}
\usepackage{latexsym}
\usepackage{mathrsfs}
\usepackage[T1]{fontenc}
\usepackage{enumitem}
\usepackage{mathrsfs}

\usepackage{amssymb}
\usepackage{latexsym}
\usepackage{amsmath}
\usepackage{fancyhdr}
\usepackage{bbm}
\usepackage{setspace}
\usepackage{mathabx}

\numberwithin{equation}{section}

\newtheorem{theorem}{Theorem}[section] 
\newtheorem{lemma}[theorem]{Lemma}     
\newtheorem{corollary}[theorem]{Corollary}
\newtheorem{proposition}[theorem]{Proposition}
\theoremstyle{definition}

\newtheorem{question}[theorem]{Question}
\newtheorem{remark}[theorem]{Remark} 

\newcommand {\C}{\mathbb C}

\newcommand {\N}{\mathbb N}
\newcommand {\Z}{\mathbb Z}
\newcommand {\K}{\mathcal K}

\newcommand {\A}{\mathcal A}
\newcommand{\rad}{{\rm rad\,}}
\newcommand{\supp}{{\rm supp\,}}

\newcommand{\eps}{\varepsilon}

\newcommand{\F}{\mathbb F}

\newcommand{\pdt}{\underline}

\title[Radicals of Biduals]{Radicals of biduals of Beurling algebras can be different for the two Arens products}

\begin{document}

	\author[J.\ T.\  White]{Jared T.\ White}
	\address{
		Jared T. White, School of Mathematics and Statistics, The Open University, Walton Hall, Milton Keynes MK7 6AA, United Kingdom.}
	\email{jared.white@open.ac.uk}
	
	
	
	\date{2026}
	
	\maketitle
	
	\begin{center}
		\textit{The Open University}
	\end{center}
	
	\begin{abstract}
		Let $\rad$ denote the Jacobson radical of a Banach algebra, and let $\Box$ and $\Diamond$ denote the two Arens products on its bidual. We give an example of a Beurling algebra $\A$ for which $\rad(\A^{**}, \Box) \neq \rad(\A^{**}, \Diamond)$, answering a question of Dales and Lau. The underlying group in our example is the free group on three generators.
	\end{abstract}
	
	\section{Introduction}
	
	A fundamental construction in the theory of Banach algebras is the Arens product, which is a multiplication defined on the bidual of a Banach algebra extending its original multiplication. Actually there are two ways to do this, so that the bidual of a Banach algebra has two (potentially different) multiplications defined on it, both known as Arens products, and denoted here by $\Box$ and $\Diamond$. 
	
	The Arens products have many important applications. Firstly, they provide a way to turn arguments about bounded left/right approximate identities in a Banach algebra $\A$ into arguments about actual left/right identities in $\A^{**}$. They have also found applications in studying amenability and related properties of groups. For example, the Arens product arguments of \cite{F90} are ultimately an important ingredient in Caprace and Monod's solution to Reiter's Problem about amenability of closed subgroups \cite{CM}. More recently, Arens products have been used directly to construct conjugation-invariant means on groups \cite{DTW}. Another application is found in combinatorics: the analogue of the Arens product for semigroup compactifications plays an important role in Ramsey Theory \cite{HS}; the relationship between the semigroup and Banach algebra Arens product is explored in, for example, \cite{DLS, DS22, DSZZ}.
	
	Arens products also provide a useful tool for distinguishing operator algebras from other types of Banach algebras. Indeed, for an operator algebra, the two Arens products always coincide, so Banach algebras whose Arens products do not coincide cannot be isomorphic to an operator algebra. This method is used for example in \cite{LS} to show that certain Beurling algebras cannot be isomorphic to operator algebras; similarly, Beurling-Fourier algebras are considered in \cite{GLSS}.

	If the two Arens products agree then the Banach algebra is called 
	\textit{Arens regular}. As mentioned, operator algebras are always Arens regular. On the other hand $\ell^1(G)$, for a group $G$, lies at the opposite extreme: the two products (which, remember, extend the product on $\ell^1(G)$) only agree on $\ell^1(G)$ itself. Dales and Lau wondered whether intermediate cases were possible, and this thought precipitated their memoir \cite{DL} in which they explored the wide range of different behaviour that the Arens products can exhibit, drawing predominantly from the class of weighted $\ell^1$-algebras on groups for their examples, a class that we call Beurling algebras.
	
	One fundamental structural invariant of a Banach algebra is its Jacobson radical. This is defined to be the largest left (or equivalently right) ideal consisting entirely of quasi-nilpotent elements. The radical is a two-sided ideal, and is important to understanding the Banach algebra as it can be characterised as the set of those elements which are not easily understood by representation theory (precisely it is equal to the set of those elements that annihilate every simple left module). 
	In light of the importance of Arens regularity, or its failure, in applications of bidual techniques, it is natural to ask how different the Jacobson radicals corresponding to each Arens product can be.
	In their memoir Dales and Lau pose the following question \cite[Question 3, Chapter 14]{DL}.
	\begin{question} 		\label{Q1}
		Does there exist a locally compact group $G$ and a weight $\omega$ on $G$ such that the radicals of $(L^1(G, \omega)^{**}, \Box)$ and $(L^1(G, \omega)^{**}, \Diamond)$ are distinct sets? 
	\end{question}
	\noindent
	The purpose of this article is to construct an example that resolves this question in the affirmative. 
	
	A simple argument (see Lemma \ref{2.1} below) shows that, if $\A$ is a commutative Banach algebra, then $\rad(\A^{**}, \Box) = \rad(A^{**}, \Diamond)$. As such, we are forced to choose a non-abelian group for our example.	
	
	On the other hand, examples of Banach algebras $\A$ for which $\rad(\A^{**}, \Box) \neq \rad(\A^{**}, \Diamond)$ were known to Dales and Lau. For instance \cite[Example 6.2]{DL} shows that the algebra of compact operators $\K(c_0)$ has $\rad(\K(c_0)^{**}, \Box) = \{0 \}$ but $\rad(\K(c_0)^{**}, \Diamond) \neq \{ 0 \}$. Moreover, \cite[Example 6.3]{DL} gives an example of a Banach *-algebra for which the two radicals of the second dual are distinct. Ours in the first such example of a Banach algebra defined in terms of a group.	
	
	For our example, we take our group to be $G = \F_3$, the free group on three generators, to be our group. Our weight will be of the following type. Given a generating set $T$ for a discrete monoid $\mathcal{M}$ the map $\omega(x) = \exp(|x|_T) \ ( x \in \mathcal{M})$, where $|\cdot |_T$ denotes word length with respect to $T$, always defines a weight on $\mathcal{M}$, due to subadditivity of $| \cdot |_T$. Our method is to carefully define a certain infinite set $X$ which generates $\F_3$ as a monoid, such that the corresponding weight defines a Beurling algebra with the desired property. 
		
	The idea of defining a weight using an infinite generating set for a finitely-generated group in order to get interesting properties goes back at least to \cite[Example 9.17]{DL}, which the authors of that work attribute to Joel Feinstein. The present author later developed this technique to construct further examples \cite{W0, W2}. These examples all take place on $\Z$. One innovation of the present work over what has gone before is to apply this technique to a non-abelian group. As a consequence the word length calculations are much more involved. 
	
	We note that our weight $\omega$ is not symmetric, and as such $\ell^1(\F_3, \omega)$ is not a Banach *-algebra. It remains open to answer Question \ref{Q1} for a symmetric weight on the group.
	
	The paper is structured as follows. In Section 2 we introduce the necessary background from group theory and Banach algebras. Section 3 gives an outline of the proof of our main theorem. Sections 4 and 5 are concerned with proving that certain elements of $\F_3$ have particular word lengths with respect to our non-standard generating set $X$. The proof of our main theorem is given in Section 6. The reader who is mainly interested in Banach algebras, and is willing to take the word length evaluations in \eqref{eqstar} and \eqref{eqdagger} below on faith, could safely skip Sections 4 and 5. The Banach algebra aspects of the proof are entirely contained in Section 6.

	\section{Background and Notation}
	\subsection{Groups and Basic Notation}	
	We shall write $\F_3 = \langle a, b, c \rangle $ for the free group on three generators. 	
	We shall refer to $a,b,c,a^{-1},b^{-1}, c^{-1}$ as letters. By a \textit{word} we mean a finite string of letters, and a word is \textit{reduced} if within the word $a$ and $a^{-1}$ never appear next to each other, and neither do $b$ and $b^{-1}$, or $c$ and $c^{-1}$. We denote the empty word by $1$, and we consider it to be reduced. 
	By successively deleting instances of $aa^{-1}, a^{-1}a, bb^{-1},$ etc., every word may be `reduced' to give a unique reduced word (although the sequence of deletions is not unique).
	The free group $\F_3$ consists of reduced words, and the group operation is given by concatenation followed by reducing the resulting word. The identity element is the empty word $1$.
		
	If $u_1, u_2, \ldots u_n \in \F_3$ then we sometimes write `$u_1u_2 \ldots u_3$ (reduced)' to indicate that no cancellations occur between letters belonging to adjacent terms $u_i$ in the product.		
	Given two reduced words $u,v \in \F_3$, we shall write $u|v$ to mean that $u$ is a subword of $v$. 
	
	Let G be a group and $T \subseteq G$. We say that $T$ is \textit{symmetric} if $t \in T$ implies that $t^{-1} \in T$. We say that \textit{$T$ generates $G$} if every $g \in G \setminus \{ 1 \}$ may be written as a product of elements from $T$. Note that we do not assume that $T$ is symmetric, so that in this article `$T$ generates $G$' always means `$T$ generates $G$ as a monoid' (rather than as a group).  If $T$ is a generating set for $G$ we define the \textit{$T$-word length} of $g \in G$ to be 
	$$|g|_T = \min \{ n \in \N : g = t_1 t_2 \cdots t_n, \text{ for some } t_1, \ldots, t_n \in T \}.$$
	If $t_1,t_2 \ldots, t_n \in T$ have the property that $g = t_1t_2\cdots t_n$ cannot be expressed using fewer than $n$ elements of $T$ (i.e. if $n = |g|_T$) then we say that the expression $t_1t_2 \dots t_n$ is \textit{$T$-minimal}. 
	
	We write $S = \{ a, b, c, a^{-1}, b^{-1} , c^{-1}\}$ for the standard generating set of $\F_3$. For brevity we shall simply write $|u|$ for $|u|_S$, where $u \in \F_3$, and similarly we refer to $|u|$ simply as the word length of $u \in \F_3$, rather than the $S$-word length.	
	
	Since we shall deal with many lengthy products of elements of $\F_3$ we find it convenient to introduce the following notation: given $d_1,\ldots, d_k \in S$ we shall write $\pdt{d} = d_1d_2\cdots d_k$. Sometimes we may have many products which themselves require indexing. In this case we shall write them as $\pdt{d}_1, \ldots, \pdt{d}_n$, where $\pdt{d}_i = d_{i, 1}d_{i,2}\cdots d_{i,p_i}$, where $p_1, \ldots, p_n \in \N$, and where $d_{ij} \in S$ for all $i,j$. Note that the product is then taken over the second index. Usually we shall suppress the notation $p_i$ representing the number of elements in each product. We interpret an empty product of group elements to be $1$. 
	
	Sometimes, given a sequence $d_1, \ldots, d_k \in S$, and $1<l<k$, we may want to split the sequence in two at $l$, in order to obtain two products, and we may write e.g. $d_i' = d_i \ (i =1, \ldots, l-1)$ and $d_i'' = d_{i-1+l} \ (i=1, \ldots, k-l+1)$, so that $\pdt{d} = \pdt{d}' \pdt{d}''$ (reduced). For the ease of the reader, we shall usually define the elements of the sequences $d'$ and $d''$ implicitly via their product, rather than giving formulae for $d_i'$ and $d_i''$ as above, writing e.g. `let $\pdt{d}'' = d_l d_{l+1}\cdots d_k$'.
	
	We fix the following notation for specific homomorphisms from the free group $\F_3$ to the integers.	Let $\theta \colon \F_3 \to \Z$ be the homomorphism defined by $\theta(a) = \theta(b) = 1$, and $\theta(c) = 0$. Let $\varphi_a \colon \F_3 \to \Z$ be the homomorphism defined by $\varphi_a(a) = 1, \ \varphi_a(b) = \varphi_a(c) = 0$, and define $\varphi_b$ and $\varphi_c$ analogously. In fact $\theta = \varphi_a + \varphi_b$.

	\subsection{Banach Algebras}
	Let $\A$ be a Banach algebra. The two Arens products on $\A^{**}$, denoted by $\Box$ and $\Diamond$, are defined in three stages: first we define the action of $\A$ on $\A^*$; next we define $\Phi \cdot \lambda$ and $\lambda \cdot \Psi$ for $\lambda \in \A^*$ and $\Phi, \Psi \in \A^{**}$; finally we define $\Phi \Box \Psi$ and $\Phi \Diamond \Psi$. The formulae are as follows:
	\begin{align*}
		\langle \lambda \cdot a, b \rangle  &= \langle \lambda, ab \rangle, &
		\langle a \cdot \lambda, b \rangle &= \langle \lambda, ba \rangle, \\
		\langle \Phi \cdot \lambda, a \rangle &= \langle \Phi, \lambda \cdot a \rangle, &
		\langle \lambda \cdot \Psi, a \rangle &= \langle \Psi, a \cdot \lambda \rangle, \\
		\langle \Psi \Box \Phi, \lambda \rangle &= \langle \Psi, \Phi \cdot \lambda \rangle, &
		\langle \Psi \Diamond \Phi, \lambda \rangle &= \langle \Phi, \lambda \cdot \Psi \rangle, 
	\end{align*}
	for $\Phi, \Psi \in \A^{**}, \lambda \in \A^*, a, b \in \A$. Both of these products extend the multiplication on $\A$. The $\Box$ product is weak*-continuous on the right, whereas the $\Diamond$ product in weak*-continuous on the left. This fact gives us formulae for the Arens products as repeated weak*-limits: let $\Phi, \Psi \in \A^{**}$, and let $(a_\alpha)$ and $(b_\beta)$ be nets in $\A$ that weak*-converge to $\Phi$ and $\Psi$ respectively. Then 
	\begin{equation}		\label{eqArens}
	\Phi \Box \Psi = \lim_\alpha \lim_\beta a_\alpha b_\beta, 
	\qquad
	\Phi \Diamond \Psi = \lim_\beta \lim_\alpha a_\alpha b_\beta,
	\end{equation}
	where both limits are taken in the weak*-topology.
	
	The Arens products were introduced by Arens in \cite{Ar1, Ar2}. For further information about Arens products, \cite[Chapter 1.4]{Palmer} contains a very readable account of the basic properties, and \cite{DL} gives a detailed introduction and background.
	
	Let $G$ be a group with identity $1$. A function $\omega \colon G \to (0, \infty)$ is called a \textit{weight on $G$} if $\omega(1) = 1$ and $\omega(st) \leq \omega(s)\omega(t) \ (s,t \in G)$. We define 
	$$\ell^1(G, \omega) = \{ f \colon G \to \C : \sum_{t \in G} |f(t)| \omega(t) <\infty \}.$$
	This is a Banach space with the norm $\|f\|_\omega = \sum_{t \in G} |f(t)| \omega(t)$. Given $f,g \in \ell^1(G, \omega)$ we define the convolution of $f$ and $g$ to be 
	$$(f*g)(t) = \sum_{s \in G} f(s)g(s^{-1}t) \quad (t \in G).$$
	It can be checked that $f*g$ again belongs to $\ell^1(G, \omega)$ and that $\|f*g\|_\omega \leq \|f\|_\omega \|g\|_\omega$, so that $\ell^1(G, \omega)$ with convolution is a Banach algebra. We call a Banach algebra of this form a \textit{Beurling algebra}. Its dual space may be identified with 
	$$\ell^1(G,\omega)^* = \ell^\infty(G, 1/\omega) = \{ \lambda \colon G \to \C : \lambda/\omega \in \ell^\infty(G) \}.$$
	
	Given a group $G$ and a generating set $T$ for $G$ we can define a weight $\omega$ on $G$ by the formula 
	$$\omega(g) = \exp(|g|_T) \quad (g \in G).$$
	The fact that this is submultiplicative follows from the 
	subadditivity of $|\cdot |_T$. It is a weight of this form on $\F_3$ that we shall use for our main example.
	
	Let $\A$ be a unital algebra and denote the spectrum of an element $a \in \A$ by $\sigma(a)$. We say that $a \in \A$ is \textit{quasi-nilpotent} if $\sigma(a) = \{ 0 \}$; by the spectral radius formula, this is equivalent to asking that $\lim_{n \to 0} \| a^n \|^{1/n} = 0$. Nilpotent elements are quasi-nilpotent. The \textit{Jacobson Radical of $\A$}, hereafter denoted by $\rad \A$, is defined to be
	the largest left  ideal consisting of quasi-nilpotent elements, that is
	$$\rad \A = \{ a \in \A : \sigma(ba) = \{0\} \ (b \in \A) \}.$$
	The Jacobson radical is always a closed two-sided ideal of $\A$, and is also the largest right ideal consisting of quasi-nilpotent elements:
	$$\rad \A = \{ a \in \A : \sigma(ab) = \{0\} \ (b \in \A) \}.$$
	In particular any $a \in \A$ for which $\A a$ is nilpotent belongs to $\rad \A$, a fact that we shall make use of in Section 6 to check that the element we construct really belongs to the radical.

	\subsection{Nets and Subnets}
	A \textit{directed set} is a set $\Lambda$  with a relation $\leq$ on $\Lambda$ which is reflexive and transitive, and for which every two elements have a common upper bound. 
	Let $X$ be a topological space. Then a \textit{net} in $X$ is a function on a directed set taking values in $X$; we typically denote a net on the directed set $\Lambda$ by $(x_{\alpha})_{\alpha \in \Lambda}$, or simply as $(x_\alpha)$. 
	A net $(x_\alpha)_{\alpha \in \Lambda}$ converges to a point $x \in X$ if, given an open neighbourhood $U$ of $x$, there exists $\alpha_0 \in \Lambda$ such that $x_\alpha \in U$ for all $\alpha \geq \alpha_0$.
	We use the word `subnet' in the sense of Willard; be aware that there are other (inequivalent) definitions of `subnet' in use.
	Precisely, if $(x_\alpha)_{\alpha \in \Lambda}$ is a net in $X$, then for us a \textit{subnet} of $(x_\alpha)_{\alpha \in \Lambda}$ is another net $(y_\beta)_{\beta \in M}$ in $X$, together with an order-preserving function $h \colon M \to \Lambda$, with the following properties:
	\begin{enumerate}
	\item $y_\beta = x_{h(\beta)}$ for all $\beta \in M$;
	\item $h$ is cofinal, meaning that for every $\alpha \in \Lambda$ there exists $\beta \in M$ such that 
	$h(\beta) \geq \alpha$.	
	\end{enumerate}
	 We often denote a subnet of a sequence $(x_k)_{k \in \N}$ as $(x_{h(\beta)})_{\beta \in M}$, where $h$ is some unspecified function from a directed set $M$ to $\N$, which is order-preserving and cofinal. It is a standard result that a topological space $X$ is compact if and only if every net has a convergent subnet.
	
	\subsection{The Stone--{\v C}ech Compactification}	 
	Given a set $X$ we write $\beta X$ for the Stone--{\v C}ech compactification of $X$.  In fact $\ell^1(X)^{**}$ may be identified isometrically with $M(\beta X)$, the Banach space of complex regular Borel measures on $\beta X$. If $G$ is a group (or more generally a semigroup), identifying $\beta G$ with the set of point-masses inside $M(\beta G)$ makes $\beta G$ into a semigroup for either Arens product. 
	
	Let $G$ be a group, and $\omega$ a weight on $G$. Even in the weighted setting Stone--{\v C}ech compactifications appear as important special subsets of $\ell^1(G, \omega)^{**}$, as we shall now explain. Following Dales and Dedania \cite{DD}, we define $\beta G_\omega$ to be the weak*-closure inside $\ell^1(G, \omega)^{**}$ of the set of normalised point masses $\{ \frac{1}{\omega(g)} \delta_g : g \in G \}$. The linear map $L_\omega \colon \ell^1(G) \to \ell^1(G, \omega)$ given by $L_\omega(f) = f/\omega$ is a surjective isometry, and the map $L_\omega^{**}|_{\beta G} \colon \beta G \to \beta G_\omega$ is a homeomorphism. Note however that $L_\omega$ does not usually respect convolutions, and that $\beta G_\omega$ is not naturally a semigroup unless $\omega = 1$. The functionals in our main example will belong to $(\beta \F_3)_\omega$, and we shall say more about this in Remark \ref{DDRemark}.

	\section{Outline of the Proof}
	\noindent
	Our aim is to construct a weight $\omega$ on $\F_3$ such that the radicals of $\ell^1(\F_3, \omega)^{**}$ are different for the two Arens products. Our method will be to define a certain infinite set $X$,  which generates $\F_3$ as a monoid, and then define our weight as $\omega(t) = \exp(|t|_X) \ (t \in \F_3)$. 
	
	We define the generating set $X$ as follows. For each $j \in \N$ let
	$$X_j = \{ v b^{5^{2j-1}}v^{-1} a^{5^{2j}}b^{5^{2j}} : v \in \F_3, \  |v| \leq 5^j \}$$
	and let 
	$$X_\infty = \bigcup_{j = 2}^\infty X_j.$$
	Finally, our generating set is 
	$$X = X_\infty \cup S.$$
	We shall write $|u|_X$ for the word length of $u \in \F_3$ with respect to the new generating set $X$.
	Of course, we always have $|u|_X \leq |u|$ for any $u \in \F_3$. Given $j \in \N$ and $v \in \F_3$ with $|v| \leq 5^j$, we shall write 
	$x(v,j) = v  b^{5^{2j-1}}v^{-1} a^{5^{2j}}b^{5^{2j}},$
	so that $X_j = \{ x(v,j) : |v| \leq 5^j\}$.
	
	Our strategy is then to define elements $\Phi_0, \Psi_0 \in \ell^1(\F_3, \omega)^{**}$ with the following properties:
	\begin{enumerate}
		\item writing $I = \ell^1(\F_3, \omega)^{**}\Box \Phi_0$, we have $I^{\Box 2} = \{ 0 \}$, which implies that
		 $$\Phi_0 \in \rad (\ell^1(\F_3, \omega)^{**}, \Box);$$		
		\item $\Phi_0 \Diamond \Psi_0$ is not quasi-nilpotent, so that $$\Phi_0 \notin \rad( \ell^1(\F_3, \omega)^{**}, \Diamond).$$
	\end{enumerate}
	The functional $\Phi_0$ will be defined as a weak*-accumulation point of the sequence
	$$\frac{1}{\omega(a^{5^{2n}}b^{5^{2n}})} \delta_{a^{5^{2n}}b^{5^{2n}}} \quad (n \in \N).$$
	In order to prove that it satisfies condition (1) above, we shall need to prove the following key word length evaluation:
	\begin{equation}		\label{eqstar}
		|a^{5^{2n}}b^{5^{2n}}|_X =  5^{2n-1} +1 \quad (n \in \N).
	\end{equation}
	This will be done in Proposition \ref{3.5}(i). Showing that $|a^{5^{2n}}b^{5^{2n}}|_X \leq 5^{2n-1} +1$ is one line (just note that $a^{5^{2n}}b^{5^{2n}} = (b^{-1})^{5^{2n-1}}x(1,n)$), but showing the reverse inequality requires quite a bit of work. 	
	Similarly, $\Psi_0$ will be defined as a weak*-accumulation point of the sequence 
	$$\frac{1}{\omega(c^n)} \delta_{c^n} = \delta_{c^n} \quad (n \in \N).$$
	In order to show that condition (2) above holds, the key calculation is the following:
	\begin{equation} 		\label{eqdagger}
		\left|a^{5^{2n_1}} b^{5^{2n_1}} c^{k_1} \cdots a^{5^{2n_r}}b^{5^{2n_r}} c^{k_r} \right|_X =  \sum_{i=1}^r (5^{2n_i-1} +1) + \sum _{i = 1}^r k_i,
	\end{equation}
	whenever $k_{i-1} > 3 \cdot 5^{n_i} \ (i = 2,3, \ldots, r)$. Together with Equation \eqref{eqstar} and Lemma \ref{3.1c} below, this will imply that 
	\begin{equation} 	\label{eqdoubledagger}
		\left|a^{5^{2n_1}} b^{5^{2n_1}} c^{k_1} \cdots a^{5^{2n_r}}b^{5^{2n_r}} c^{k_r} \right|_X  = \sum_{i=1}^r |a^{5^{2n_i}}b^{5^{2n_i}}|_X + \sum_{i=1}^r |c^{k_i}|_X
	\end{equation}
	whenever $k_{i-1} > 3 \cdot 5^{n_i} \ (i = 2,3, \ldots, r)$.
	Again,  proving that the left-hand side of \eqref{eqdagger} is at most the right-hand is straight-forward, and it is the reverse inequality that requires work.
	
	Section 4 and Section 5 will be devoted to proving \eqref{eqstar} and \eqref{eqdagger}. Specifically \eqref{eqstar} follows from Proposition \ref{3.5}(i) and \eqref{eqdagger} is Proposition \ref{3.8}. Then, in Section 6, $\Phi_0$ and $\Psi_0$ will be formally introduced, and conditions (1) and (2) will be proved. This will establish our main result (Theorem \ref{4.6}). 
	
	We now briefly discuss why we chose the particular group and weight that we did for our main example. First of all, the following lemma shows that we have to search for our example amongst non-abelian groups.
	
	\begin{lemma} 	\label{2.1}
		Let $\A$ be a commutative Banach algebra. Then $\rad(\A^{**}, \Box) = \rad(\A^{**}, \Diamond)$.
	\end{lemma}
	
	\begin{proof}
		Let $\Phi \in \rad( \A^{**}, \Box)$. Then $\Phi \Box \A^{**}$ consists of quasi-nilpotent elements. Letting $\Psi \in \A^{**}$, we note that, because $\A$ is commutative, $\Diamond$ is the opposite product of $\Box$ (as can be seen from \eqref{eqArens}), and therefore
		$\|(\Psi \Diamond \Phi )^{\Diamond n} \|^{1/n} =  \| (\Phi \Box \Psi)^{\Box n} \|^{1/n} \to 0,$
		as $n \to \infty$. We have shown that $\A^{**} \Diamond \Phi$ consists of quasi-nilpotent elements, and hence $\Phi$ is radical with respect to $\Diamond$.		
		This proves that $\rad( \A^{**}, \Box) \subseteq \rad(\A^{**}, \Diamond)$, and the reverse inclusion is proved analogously.
	\end{proof}
	
	It remains an open question whether or not the two radicals of $\ell^1(G)^{**}$, for some group $G$, can be distinct sets. However, at the time of writing the only known method of constructing radical elements of $\ell^1(G)^{**}$ involve using invariant means on $G$ and its subgroups, and the resulting elements are always radical for both Arens products. It is a long standing open problem whether $\ell^1(\F_n)^{**} \ (n \geq 2)$ has a non-zero radical (for either Arens product). In the light of these factors, we decided to search for our example amongst non-trivial weights on very non-abelian groups, eventually arriving at the main example of this paper.  
	
	Although Question \ref{Q1} remains open for $\omega = 1$, there is an interesting property that our radical element has that cannot be shared by radical elements in the unweighted setting: it belongs to  $(\beta \F_3)_\omega$. See Remark \ref{DDRemark} for details.

	\section{Basic Lemmas For Calculating Word Lengths}
	\noindent
	In this section we present some basic lemmas that will be important tools for calculating $X$-word lengths (and hence values of $\omega$) in Section 5. We note that a general product of elements of $X$ has the form $\pdt{d}_1 x_1 \pdt{d}_2 \cdots \pdt{d}_n x_n \pdt{d}_{n+1}$, where $x_i \in X_\infty \ (i = 1, \ldots, n)$, and ${d_{ij} \in S} \ (i = 1, \ldots, n+1, \ j = 1, \ldots, p_i)$, where we allow for the possibility that some of the products $\pdt{d}_i$ are empty. Throughout this section and the next we shall refer to a particular way of writing an element of $\F_3$ as a product of elements of $X$ as an \textit{expression over $X$}, or sometimes as an \textit{$X$-word}. Note that, since $S \subseteq X$, if an expression $\pdt{d}_1 x_1 \pdt{d}_2 \cdots \pdt{d}_n x_n \pdt{d}_{n+1}$ is $X$-minimal, then $d_{i, 1} d_{i,2} \cdots d_{i,p_i}$ is a reduced word for each $i=1, \ldots, n+1$.
	
	The first lemma of this section is elementary.
	
	\begin{lemma} 	\label{3.1c}
		Let $k \in \N$. Then $|c^k|_X = k$.
	\end{lemma}
	
	\begin{proof}
		Certainly $|c^k|_X \leq k$. If $x_1,x_2, \ldots, x_n \in X$ and $x_1x_2 \cdots x_n = c^k$ then, noting that $\varphi_c(x_i) \leq 1$ for each $i$, we get 
		$k = \varphi_c(c^k) = \sum_{i=1}^n \varphi_c(x_i) \leq n.$
		The result follows.
	\end{proof}
	
	The next lemma will be a useful tool for calculating word lengths in the next section.
	
	\begin{lemma} 	\label{3.0}
		Let $u \in \F_3$. Then $|u| \geq |\varphi_a(u)| + |\varphi_b(u)| + |\varphi_c(u)|.$
	\end{lemma}
	
	\begin{proof}
		Let $\alpha_+$ denote the number of positive $a$ terms in $u$ as a reduced word, and let $\alpha_-$ denote the number of negative $a$ terms. Likewise define $\beta_+$ and $\beta_-$ for the $b$ terms, and $\gamma_+$ and $\gamma_-$ for the $c$ terms. Then
		\begin{equation}		\label{eqphi}
			|u| = {\alpha_+} + {\alpha _-} + {\beta_+} + {\beta_-} + {\gamma_+} + {\gamma_-}
		\end{equation}
		and
		$$\varphi_a(u) = \alpha_+ - \alpha_-, \qquad \varphi_b(u) = \beta_+ - \beta_, \qquad \varphi_c(u) = \gamma_+- \gamma_-.$$
		It follows that $\alpha_+ \geq \alpha_+ - \alpha_- = \varphi_a(u)$, and likewise $\alpha_- \geq -\varphi_a(u)$, so that $\alpha_+ + \alpha_- \geq |\varphi_a(u)|.$ Similarly we can show that $\beta_+ + \beta_- \geq |\varphi_b(u)|$, and $\gamma_+ + \gamma_- \geq |\varphi_c(u)|.$ The result now follows from \eqref{eqphi}.
	\end{proof}
	
	The following lemma is a variation on a well-known result, and its corollary tells us how to write the generators from $X_\infty$ as reduced words.
	
	\begin{lemma} 		\label{3.1}
		Let $k \in \N$ and let $v \in \F_3$ with $|v| \leq k$. Then $vb^kv^{-1} = w b^k w^{-1}$, where $w$ is a reduced word ending in $a^{\pm 1}$ or $c^{\pm 1}$, and $|w| \leq k $. No further cancellations can occur between $w$, $w^{-1}$ and $b^k$.
	\end{lemma}
	
	\begin{proof}
		The case $v = 1$ is trivial. Otherwise we can write $v = wb^j$ for some $j \in \Z$, where $w$ ends in $a^{\pm 1}$ or $c^{\pm 1}$. Then $v^{-1} = b^{-j}w^{-1}$, where $w^{-1}$ begins in $a^{\pm 1}$ or $c^{\pm 1}$, so that no cancellation can occur in $b^kw^{-1}$. We then have
		$$vb^kv^{-1} = wb^j b^k b^{-j} w^{-1} = wb^k w^{-1},$$
		and the final expression is reduced.
	\end{proof}
	
	Let $j \in \N,$ $j \geq 2$, and let $v \in \F_3$ with $|v| \leq 5^j$. From now on, when we consider the generator $x(v,j) \in X_\infty$, we may assume without loss of generality that $vb^{5^{2j-1}}v^{-1}$ is reduced, since, taking $w$ as in the previous lemma, we have $x(v,j) = x(w,j)$. If we need a formula for the whole of $x(v,j)$ as a reduced word, we have the following.
	
	\begin{corollary}		\label{3.1a}
		Let $n \in \N$ and let $v \in \F_3$ with $|v| \leq 5^n$. As a reduced word $x(v,n)$ has the form
		$a^k w b^{5^{2n-1}} w^{-1} a^{5^{2n}-k}b^{5^{2n}}$, and  $w \in \F_3$ and $k \in \N$ satisfy $|w|+k \leq 5^n$. 
	\end{corollary}
	
	\begin{proof}
		Immediate from the previous lemma.
	\end{proof}
	
	The following is an elementary but crucial property of $X$-minimal expressions that we shall use repeatedly.
	
	\begin{lemma}		\label{3.3}
		Let $j \in \N$ and let $v \in \F_3$ with $|v| \leq 5^j$. Suppose that 
		\begin{equation}	\label{eq3.2}
			\pdt{d} x(v,j) \pdt{e}
		\end{equation}
		is $X$-minimal, where $d_i, e_i \in S$. Then less than half of the letters in $x(v,j)$ (as a reduced word) are cancelled by $\pdt{d}$ and $\pdt{e}$.
	\end{lemma}
	
	\begin{proof}
		Assume otherwise. Write $\pdt{d} = d_1 d_2 \cdots d_ld_{l+1} \cdots d_{l+p}$ and 
		$\pdt{e} = e_1 e_2 \cdots e_m e_{m+1} \cdots e_{m+q}$, where $d_{l+1}, \ldots, d_{l+p}, e_1, \ldots, e_m$ are the letters that cancel, and $d_1, \ldots, d_l, e_{m+1}, \ldots e_{m+q}$ do not. Then \eqref{eq3.2} contains $p+q+m+l +1$ symbols from $X$, and by hypothesis $p+m \geq \frac{1}{2}|x(v,j)|$. Write
		$$\pdt{d}' = d_{l+1} d_{l+2} \cdots d_{l+q}, \qquad \pdt{e}' = e_1 e_2 \cdots e_m.$$
		Then, since at least half of the letters are cancelled, $\pdt{d}'x(v,j)\pdt{e}' = t_1t_2\cdots t_n$ for some $n < \frac{1}{2}|x(v,j)|$, and $t_1, \ldots, t_n \in S$. We get
		$$\pdt{d}x(v,j) \pdt{e} =  d_1d_2\cdots d_l \, t_1t_2 \cdots t_n \, e_{m+1} e_{m+2} \cdots e_{m+q},$$
		and the latter expression contains $l+q+n \leq l+q+ \frac{1}{2}|x(v,j)| \leq l+q +p + m$ symbols from $X$, which is less than the original expression \eqref{eq3.2}, contradicting that hypothesis that it is $X$-minimal.
	\end{proof}	
	
	\begin{lemma}		\label{3.2}
		Let $j \in \N$ and let $v \in \F_3$ with $|v| \leq 5^j$. Suppose that $\pdt{d} x(v, j)$ is $X$-minimal, where $d_1, \ldots, d_p \in S.$ Write 
		$vb^{5^{2j-1}}v^{-1} = w b^{5^{2j-1}} w^{-1}$ as a reduced word, as in Lemma \ref{3.1}. Suppose that $\pdt{d}$ cancels $wb^{5^{2j-1}}$. Then we may rewrite $\pdt{d}x(v,j)$ as $\pdt{e}x(1,j)$, for some 
		$e_1, \ldots, e_{q} \in S$, without increasing the number of symbols from $X$ in the expression.
	\end{lemma}

	\begin{proof}
		We can write $\pdt{d} = yz$ (reduced), where $z = [wb^{5^{2j-1}}]^{-1}$, and $|z| = |w| + 5^{2j-1}$, and $|\pdt{d}| = |y| + |z|$. We have 
		$$\pdt{d}x(v,j) = yw^{-1}a^{5^{2j}}b^{5^{2j}} = yw^{-1} b^{-5^{2j-1}} x(1,j).$$
		Let $e_1, \ldots, e_{q}$ be the letters appearing in $yw^{-1}b^{-5^{2j-1}}$ as a reduced word, so that $\pdt{e} = yw^{-1}b^{-5^{2j-1}}$.
		Then $\pdt{e}x(1,j) = \pdt{d}x(v,j)$ and
		$$|\pdt{e}| \leq |y| + |w^{-1}| + 5^{2j-1} = |y| + |z| = |\pdt{d}|.$$
		Therefore the number of symbols from $X$ in the expression $\pdt{e}x(1,j)$, which is $1+|\pdt{e}|$, is at most $1+|\pdt{d}|$, giving the result.
	\end{proof}
	
	Let $n \in \N$, let $k_1, \ldots, k_n \in \N$, and let $v_1, \ldots, v_n \in \F_3$ with $|v_i| \leq 5^{k_i} \ (i = 1, \ldots, n)$. We shall say that an expression 
	$$\pdt{d}_1 x(v_1, k_1) \cdots \pdt{d}_n x(v_n, k_n) \pdt{d}_{n+1},$$
	 where $d_{ij} \in S$, is \textit{super $X$-minimal} if
	 \begin{enumerate}
	 	\item $v_ib^{5^{2k_i-1}}v_i^{-1}$ is reduced for every $i = 1, \ldots, n$, and 
	 	\item  whenever some $\pdt{d}_i$ cancels $v_i b^{5^{2k_i-1}}$ we have $v_i = 1$.
	 \end{enumerate} By Lemma \ref{3.1} and Lemma \ref{3.2} every $X$-minimal expression is equivalent to a super $X$-minimal expression. 
	 If we instead use the form from Corollary \ref{3.1a} to write 
	 $$x(v_i, k_i) = a^{\nu_i} w_i b^{5^{2k_i - 1}} w_i^{-1} a^{5^{2k_i} - \nu_i} b^{5^{2k_i}} \text{(reduced)} \quad (i = 1, \ldots, n),$$
	 for some $\nu_1, \ldots, \nu_n \in \N$, and $w_1, \ldots, w_n \in \F_3$, then condition (2) above is equivalent to saying that, whenever some $\pdt{d}_i$ cancels $a^{\nu_i} w_i b^{5^{2k_i - 1}}$, we have $w_i = 1$ and $\nu_i = 0$. Most of our results in Section 5 include the hypothesis that the $X$-words are super $X$-minimal. We suspect that most of the results are still true without this assumption, but including it reduces the number of cases that have to be checked.
	 


	\section{Proof of the Key Word Length Results}
	\noindent	
	In this section we prove Equations \eqref{eqstar} and \eqref{eqdagger} -- the two key results about $X$-word lengths that we shall need for our main result. These are Proposition \ref{3.5}(i) and Proposition \ref{3.8}. 
	We begin with a short technical lemma that we shall need in Case 4 in the proof of Lemma \ref{3.7}.

	\begin{lemma}		\label{3.6}
		Let $j \in \N$, and let $v \in \F_3$ satisfy $|v| \leq 5^j$. Write 
		$x(v,j) = s_1s_2 \cdots s_m$ as a reduced word, where $m = |x(v,j)|$ and $s_1, \ldots, s_m \in S$.
		Suppose that $z$ is an initial block of these letters, and that $y$ is the rest (so that in particular $x(v,j) = zy$).
		Suppose further that $|z| < \frac{1}{2} |x(v,j)|$. Then $\theta(y) \geq |z|$. 
	\end{lemma}
	
	\begin{proof}
		In fact we shall prove that $\theta(y) \geq \frac{1}{2} |x(v,j)|$, which immediately gives the result. 
		By Corollary \ref{3.1a}, as a reduced word 
		$$x(v,j) = a^k w b^{5^{2j-1}} w^{-1} a^{5^{2j} - k} b^{5^{2j}}, $$
		where $|w|+ k \leq 5^j$.
		There are two cases to consider.		
		First, suppose that $y|a^{5^{2j}-k}b^{5^{2j}}$.
		Then in fact $\theta(y) = |y|$, which is at least 
		$\frac{1}{2} |x(v,j)|$ by hypothesis.
		
		The only other possibility is that $y$ ends  
		$a^{5^{2j}-k}b^{5^{2j}}$, but also contains some earlier terms. The only negative contributions to $\theta(y)$ can come from $w$ and $w^{-1}$, so that a crude lower bound is given to us by
		\begin{align*}
			\theta(y) &\geq \theta(a^{5^{2j}-k}b^{5^{2j}}) - 2|w| = 2\cdot 5^{2j} -k - 2|w| \\
			&\geq 2 \cdot 5^{2j} - 2\cdot 5^j \geq 5^{2j} + 5^{2j-1} + 5^j 
			\geq \tfrac{1}{2} |x(v,j)|,
		\end{align*}
		as required.
	\end{proof}

	\begin{lemma}		\label{3.7} 
		Let $n \in \N$ and let $x_1, \ldots, x_n \in X_\infty$. Suppose that the expression
		\begin{equation}	\label{eq3.15}
			\pdt{d}_1x_1 \pdt{d}_2 \cdots \pdt{d}_n x_n
		\end{equation}
		is super $X$-minimal, where $x_i \in X_\infty$ and $d_{ij} \in S$ for all $i,j$. Write 
		$$x_i = a^{\nu_i}w_ib^{5^{2k_i-1}}w_i^{-1}a^{5^{2k_i} - \nu_i}b^{5^{2k_i}} \ \text{(reduced)} \quad (i = 1, \ldots, n),$$
		as in Corollary \ref{3.1a}, where $\nu_1, \ldots, \nu_n \in \N_0$, $w_1, \ldots, w_n \in \F_3$, and $k_1, \ldots, k_n \in \N$.			
		Fix a sequence of cancellations that reduces the expression \eqref{eq3.15}.
		Then if any terms from $a^{5^{2k_n} - \nu_n}$ are cancelled, then they are cancelled by letters from $\pdt{d}_n$.	
	\end{lemma}

	\begin{proof}	
		Suppose that $r$ of the letters from $a^{5^{2k_n} - \nu_n}$ are cancelled. Then there must be some contiguous block of letters at the end of the string
		\begin{equation}		\label{eq3.14}
			\pdt{d}_1x_1 \pdt{d}_2 \cdots x_{n-1} \pdt{d}_n \, a^{\nu_n} w_n b^{5^{2k_n-1}} w_n^{-1}
		\end{equation}
		that reduces to $a^{-r}$. Our proof shall proceed by dividing into cases according to the various possibilities for where this block can lie. Only Case 1 and Subcase 4C at the very end do not result in a contradiction.
		
		\vskip 2mm
		\noindent
		\underline{Case 1.} The block comes entirely from $\pdt{d}_n a^{\nu_n} w_n b^{5^{2k_n-1}} w_n^{-1}$.
		
		In this case, we see that $\pdt{d}_n$ must cancel $b^{5^{2k_n-1}}$, and hence (as \eqref{eq3.15} is super $X$-minimal) $w_n=1$ and $\nu_n = 0$. It is then clear that $\pdt{d}_n = y a^{-r}b^{-5^{2k_n-1}}$ (reduced), for some $y \in \F_3$, and  in particular that the cancellations all come from the $a^{-1}$'s in $\pdt{d}_n$.
		
		\vskip 2mm
		\noindent
		\underline{Case 2.} The block ends somewhere within $\pdt{d}_m$, for some $m<n$. 
		
		
		
		In this case, divide the sequence $d_m$ into into an initial part $d_m'$, and the rest $d_m''$, so that $\pdt{d}_m''$ is the part of $\pdt{d}_m$ that belongs to the block. We have
		$$ a^{-r} = \pdt{d}_m''x_m \pdt{d}_{m+1} \cdots x_{n-1} \pdt{d}_n \, a^{\nu_n} w_n b^{5^{2k_n-1}} w_n^{-1}.$$
		Taking $\varphi_a$ of both sides, and noting that $\varphi_a(x_i) > 1$ for all $i$, we have
		\begin{align*}
			-r &= \varphi_a(\pdt{d}_m'')  +\sum_{i = m+1}^n \varphi_a(\pdt{d}_i) 
			+\sum_{i=m}^{n-1} \varphi_a(x_i) + \nu_n\\
			&> \varphi_a(\pdt{d}_m'')  +\sum_{i = m+1}^n \varphi_a(\pdt{d}_i) + (n-m) +\nu_n,
		\end{align*}
		which implies that
		$$ -\varphi_a(\pdt{d}_m'')  - \sum_{i = m+1}^n \varphi_a(\pdt{d}_i) > r +(n-m) +\nu_n.$$
		Similarly
		$$ -\varphi_b(\pdt{d}_m'')  -\sum_{i = m+1}^n \varphi_b(\pdt{d}_i) > (n-m) + 5^{2k_n-1}.$$
		By Lemma \ref{3.0} these last two equations together imply that  
		\begin{align*}
			|\pdt{d}_m''| + \sum_{i=m+1}^n |\pdt{d}_i | &\geq
			|\varphi_a(\pdt{d}_m'')| + |\varphi_b(\pdt{d}_m'')| + \sum_{i=m+1}^n |\varphi_a(\pdt{d}_i)| +  \sum_{i=m+1}^n |\varphi_b(\pdt{d}_i)|\\
			&\geq r+ \nu_n  + 5^{2k_n-1} + 2(n-m)> r + \nu_n +  5^{2k_n-1} + 1.
		\end{align*}
		The $X$-word
		$$\pdt{d}_m''x_m \pdt{d}_{m+1} \cdots \pdt{d}_n x_n$$
		is equal to $a^{5^{2k_n}-r - \nu_n}b^{5^{2k_n}}$ as a reduced word, and contains 
		$$(n-m) + |\pdt{d}_m''| + \sum_{i=m+1}^n |\pdt{d}_i | \geq 2 + |\pdt{d}_m''| + \sum_{i=m+1}^n |\pdt{d}_i | > r + \nu_n + 5^{2k_n-1} + 3$$
		many symbols from $X$,
		and hence replacing this string in \eqref{eq3.15} by
		$$(a^{-1})^{r +\nu_n} (b^{-1})^{5^{2k_n-1}} x(1,k_n)$$
		would allow us to represent the same element of $\F_3$ whilst using fewer elements of $X$, a contradiction.
		
		\vskip 2mm
		\noindent
		\underline{Case 3.} The block ends either between $x_{m-1}$ and $\pdt{d}_m$, or between $\pdt{d}_m$ and $x_m$, for some $m<n$.
		
		In this case, repeat the argument of Case 2., with $\pdt{d}_m'' = \pdt{d}_m$ and $\pdt{d}_m' = 1$, or $\pdt{d}_m'' = 1$ and $\pdt{d}_m' = \pdt{d}_m$, respectively.
		
		\vskip 2mm
		\noindent
		\underline{Case 4.} The block ends part way through $x_m$, for some $m<n$.
		
		Write 
		$$x_m = z_1z_2 \cdots z_R \, y_1 y_2 \cdots y_Q \ \text{(reduced)},$$
		for some $Q, R \in \N$ such that $y_1 y_2 \cdots y_Q$ is part of the block, and $z_1 z_2 \cdots z_R$ isn't. 
		The expression $x_m \pdt{d}_{m+1} \cdots \pdt{d}_n x_n$ has $n-m + 1 + \sum_{i=m+1}^n |\pdt{d}_i|$ symbols from $X$ and as an element of $\F_3$ is equal to $\pdt{z} a^{5^{2k_n}- r- \nu_n}b^{5^{2k_n}}$.
		
		We have
		$$\pdt{y} \, \pdt{d}_{m+1} x_{m+1} \cdots \pdt{d}_n  \, w_n b^{5^{2k_n-1}} w_n^{-1} = a^{-r - \nu_n}.$$
		Applying $\varphi_a$ to both sides gives
		$$\varphi_a(\pdt{y})  + \sum_{i=m+1}^n \varphi_a(\pdt{d}_i) + \sum_{i=m+1}^{n-1} \varphi_a(x_i)  = -r - \nu_n,$$
		and hence
		\begin{equation}		\label{eq3.16}
			\sum_{i=m+1}^n \varphi_a(\pdt{d}_i) = -r -\nu_n -\varphi_a(\pdt{y}) -\sum_{i=m+1}^{n-1} \varphi_a(x_i) \leq -r -\nu_n - \varphi_a(\pdt{y}). 
		\end{equation} 
		Similarly 
		\begin{equation}		\label{eq3.17}
			\sum_{i=m+1}^n \varphi_b(\pdt{d}_i) \leq -5^{2k_n-1} - \varphi_b(\pdt{y}).
		\end{equation}
		Using Lemma \ref{3.0} and noting that $\varphi_a(\pdt{y}), \varphi_b(\pdt{y}) \geq 0$, Equations \eqref{eq3.16} and \eqref{eq3.17} together imply that 
		\begin{equation}		\label{eq3.13}
			\sum_{i=m+1}^n |\pdt{d}_i| \geq \varphi_a(\pdt{y}) + \varphi_b(\pdt{y})  +r +\nu_n + 5^{2k_n-1} = \theta(\pdt{y})  +r +\nu_n + 5^{2k_n-1}.
		\end{equation}
		We now divide into subcases depending on whether or not $|\pdt{y}|\leq \frac{1}{2}|x_m|$.

		\vskip 2mm
		\noindent
		\underline{Subcase 4A.} $|\pdt{y}| > \frac{1}{2}|x_m|$.
		
		In this case Lemma \ref{3.6} tells us that $\theta(\pdt{y}) \geq |\pdt{z}|$, and so \eqref{eq3.13} implies that 
		$$\sum_{i=m+1}^n |\pdt{d}_i| \geq |\pdt{z}| + 5^{2k_n-1} + r +\nu_n = R  + r +\nu_n + 5^{2k_n-1} ,$$
		so that the expression $x_m \pdt{d}_{m+1} \cdots \pdt{d}_n x_n$ has at least 
		$$2 + \sum_{i=m+1}^n |\pdt{d}_i| \geq 2+ R + r +\nu_n + 5^{2k_n-1}$$
		symbols from $X$. 	
		On the other hand, the expression
		$$z_1 z_2 \cdots z_R \, (a^{-1})^{r+\nu_n} (b^{-1})^{5^{2k_n-1}} x(1, k_n)$$
		also equals $\pdt{z} a^{5^{2k_n}-r - \nu_n}b^{5^{2k_n}}$, but has only $1+ R + r +\nu_n + 5^{2k_n-1}$ symbols from $X$, a contradiction.
			
		\vskip 2mm
		\noindent
		\underline{Subcase 4B.} $|\pdt{y}| \leq \frac{1}{2}|x_m|$  and $m<n-1$.
		
		This means that $\pdt{y} | a^{5^{2k_m} - \nu_m} b^{5^{2k_m}},$ so that $\theta(\pdt{y}) = |\pdt{y}|$. As such, by \eqref{eq3.13},
		$$ \sum_{i=m+1}^n |\pdt{d}_i| \geq |\pdt{y}| + r + \nu_n + 5^{2k_n-1} = Q + r + \nu_n + 5^{2k_n-1}  ,$$
		and it follows that the expression $x_m \pdt{d}_{m-1} \cdots \pdt{d}_n x_n$ contains at least 
		$$n-m + 1 + \sum_{i = m-1}^n|\pdt{d}_i| \geq  3 + Q + r + \nu_n + 5^{2k_n-1}$$
		symbols from $X$.
		Therefore, the expression
		$$x_m y_Q^{-1} \cdots y_2^{-1} y_1^{-1}(a^{-1})^{r + \nu_n} (b^{-1})^{5^{2k_n-1}} x(1, k_n)$$
		has fewer symbols from $X$ than $x_m \pdt{d}_{m+1} \cdots \pdt{d}_n x_n$, whilst representing the same group element. This is a contradiction. 
		
		\vskip 2mm
		\noindent
		\underline{Subcase 4C.} $|\pdt{y}| \leq \frac{1}{2}|x_m|$  and $m=n-1$.
		
		This means that our block is $\pdt{y}\pdt{d}_n a^{\nu_n} w_n b^{5^{2k_n-1}} w_n^{-1}$, where again  $\pdt{y} | a^{5^{2k_m} - \nu_m} b^{5^{2k_m}}$. 
		We have 
		$$\pdt{y} \pdt{d}_n a^{\nu_n} w_n b^{5^{2k_n-1}} w_n^{-1} = a^{-r}$$
		and for this product to be equal to $a^{-r}$ something must cancel $\pdt{y}$. If terms from $w$ cancel $\pdt{y}$ then they must cancel $\pdt{d}_n$ first, meaning that there is nothing available to cancel $b^{5^{2k_n-1}}$. 
		Hence the cancellation must come from $\pdt{d}_n$ and we can write $\pdt{d}_n = \pdt{d}_n' \pdt{d}_n''$ (reduced), 
		where $\pdt{d}_n' = y_Q^{-1}\cdots y_2^{-1} y_1^{-1}$. Then  $\pdt{d}_n'' a^{\nu_n} w_n b^{5^{2k_n-1}} w_n^{-1} = a^{-r}$. Now we may argue as in Case 1 to get the desired conclusion.
	\end{proof}
	
	\begin{lemma}		\label{3.4}
		Let $n \in \N$ and suppose that 
		\begin{equation}		\label{eq3.1}
			\pdt{d}_1 x_1 \cdots \pdt{d}_n x_n \pdt{d}_{n+1}
		\end{equation}
		is super $X$-minimal, where $x_1, \ldots, x_n \in X_\infty$, and $d_{ij} \in S$. 
		Further, write 
		$$x_n = a^k w b^{5^{2j-1}} w^{-1} a^{5^{2j}-k}b^{5^{2j}} \ \text{(reduced)},$$
		 as in Corollary \ref{3.1a}, for some $j,k \in \N$, and $w \in \F_3$, with $k + |w| \leq 5^{j}$, and write let $p$ denote the number of letters in $\pdt{d}_{n+1}$. 
		  Fix a sequence of cancellations that reduces \eqref{eq3.1}. Then not all of the terms from $a^{5^{2j}}$ are cancelled, and hence the product \eqref{eq3.1} ends $a^{m_1} b^{m_2} d_{n+1, p'} \cdots d_{n+1, p}$ as a reduced word, where $1 \leq p' \leq p$, $m_1>0$, and $m_2 \geq 0$.
	\end{lemma}
	
	\begin{proof}
		By Lemma \ref{3.7} terms from $a^{5^{2j}-k}$ are only cancelled by terms from $\pdt{d}_n$ and $\pdt{d}_{n+1}$.		
		In order for such cancellations to occur, either $b^{5^{2j}}$ must be cancelled from the right, or all of $a^kwb^{5^{2j-1}}w^{-1}$ must be cancelled from the left.
		
		Suppose that $b^{5^{2j}}$ is cancelled from the right. By Lemma \ref{3.3} $\pdt{d}_n$ and $\pdt{d}_{n+1}$ together cannot cancel more than 
		$\frac{1}{2}|x(v,j)| = |w| + \tfrac{1}{2} \cdot 5^{2j-1} +  5^{2j}$
		letters from $x(v,j)$. If $\pdt{d}_{n+1}$ cancels $b^{5^{2j}}$ that means that there can be at most 
		$\frac{1}{2}|x(v,j)| - 5^{2j} = |w| + \frac{1}{2} \cdot 5^{2j-1}$
		further cancellations by $\pdt{d}_n$ and $\pdt{d}_{n+1}$. This means that at least 
		$$5^{2j} - k - |w| - \tfrac{1}{2} \cdot 5^{2j-1} \geq 5^{2j} - 5^{j} - \tfrac{1}{2}\cdot5^{2j-1} >0$$
		letters from $a^{5^{2j} - k}$ remain.
		
		We can do a similar calculation if instead  $a^kwb^{5^{2j-1}}w^{-1}$ is cancelled from the left. In this case $\pdt{d}_n$ must cancel $k+ 2|w| + 5^{2j-1}$ letters, leaving at most
		\begin{align*}
			\tfrac{1}{2}|x(v,j)| - (k+2|w| + 5^{2j-1}) &=  (|w| + \tfrac{1}{2} \cdot 5^{2j-1} +  5^{2j}) - (k+ 2|w| + 5^{2j-1}) \\
			&= 5^{2j} - k - |w| - \tfrac{1}{2} \cdot 5^{2j-1}
		\end{align*}
		further cancellations available. Therefore, at least
		$$5^{2j} - k  -(5^{2j} - k - |w| - \tfrac{1}{2} \cdot 5^{2j-1}) = \tfrac{1}{2} \cdot 5^{2j-1} + |w| > 0$$
		letters from $a^{5^{2j}-k}$ remain. This completes the proof.
	\end{proof}

	\begin{lemma}	\label{3.5a}
		Let $n \in \N$, with $n \geq 2$. Then $|a^{5^{2n}} b^{5^{2n}}|_X \leq  5^{2n-1} +1$.
	\end{lemma}
	
	\begin{proof}
		Just note that 
		$a^{5^{2n}}b^{5^{2n}} = (b^{-1})^{5^{2n-1}} x(1, n).$
	\end{proof}
	
	Part (i) of the next proposition with $k=0$ proves \eqref{eqstar}, one of the two vital word length evaluations that we shall require for our main result. Part (ii) serves as the base case for the second in Proposition \ref{3.8}.
	
	\begin{proposition}		\label{3.5}
		Let $k, k_1, k_2, n \in \N_0$ with $n \geq 2$ and $k_2 \geq 1$. Then 
		\begin{enumerate}
			\item[ \rm (i)] $|c^k a^{5^{2n}} b^{5^{2n}}|_X = k + 5^{2n-1} +1$;
			\item[ \rm (ii)] $|c^{k_1} a^{5^{2n}} b^{5^{2n}}c^{k_2}|_X =  k_1+ 5^{2n-1} +1 + k_2$.
		\end{enumerate}
	\end{proposition}
	
	\begin{proof}
		We first prove (i). Clearly an $X$-minimal expression representing $c^k a^{5^{2n}} b^{5^{2n}}$ must involve at least one element of $X_\infty$.
		Therefore, suppose that we have
		$d_1, \ldots, d_l \in S$, $x_1, \ldots ,x_q \in X$, and $j \in \N$ and $ v \in \F_3$ with $|v| \leq 5^j$, such that 
		\begin{equation} 	\label{eq3.4b}
			x_1 \cdots x_q \, x(v,j) \, \pdt{d} = c^k a^{5^{2n}} b^{5^{2n}},
		\end{equation}
		and such that $q+l+1$ is minimal. Then by Lemma \ref{3.5}
		\begin{equation}		\label{eq3.4a}
			q+l+1 \leq k+ 5^{2n-1} +1,
		\end{equation}
		and we shall show the reverse inequality in this proof.
		Without loss of generality we may assume that the expression on the left-hand side of \eqref{eq3.4b} is super $X$-minimal.
		
		By Lemma \ref{3.4}, as a reduced word the left-hand side of \eqref{eq3.4b} ends $a^{m_1}b^{m_2}d_{l'+1}\cdots d_l$, where $m_1> 0$, $m_2 \geq 0$, and $d_1, \ldots , d_{l'}$ are the terms from $\pdt{d}$ that cancel with $a^{5^{2j}}b^{5^{2j}}$. We shall consider the possibilities for $d_1, \ldots, d_{l}$. Note that comparing the left- and right-hand side of \eqref{eq3.4b}, $a^{m_1}$ must be part of $a^{5^{2n}}$. As such $d_1, \ldots, d_l \in \{ a^{\pm 1}, b^{\pm 1} \}$.
		
		Observe that $\varphi_c(x) = \varphi_c(d_i) = 0$ for all $x \in X_\infty$ and $i = 1, \ldots, l$. We have  
		\begin{equation} 	\label{eq3.4d}
			k = \varphi_c(x_1 \cdots x_q \, x(v,j) \, \pdt{d}) = \sum_{i=1}^q \varphi_c(x_i),
		\end{equation} 
		which means that at least $k$ of the symbols $x_1, \ldots, x_q$ are equal to $c$.
		
		We claim that $m_2$ cannot be zero. Indeed, this would means that $d_1, d_2, \ldots, d_{5^{2j}} = b^{-1}$, and then comparing with the right-hand side of \eqref{eq3.4b} we would have 
		$d_{l''} = d_{l''+1} = \cdots =d_{l} = b$, where $l-l'' = 5^{2n}$. In total, this brings us to at least $5^{2n}$ $b$'s plus $k$ $c$'s, meaning that $q+l \geq k+5^{2n} > k + 5^{2n-1}$.
		This contradicts \eqref{eq3.4a}.
		As such, $m_2 \neq 0$ as claimed and we are left with three possibilities. We shall show that it is the third one that occurs.
		
		\vskip 2mm
		\noindent
		\underline{Case 1.} $d_1 = d_2 = \cdots = d_{l} = b^{-1}$, and $l=l'< 5^{2j}$.
		
		In this case, the left-hand side of \eqref{eq3.4b} ends $a^{m_1} b^{5^{2j}-l}$, and comparing with the right-hand side we see that $5^{2j}-l = 5^{2n}$, and we must have $j>n$. Then 
		$$l = 5^{2j}-5^{2n} \geq 5^{2n+2} - 5^{2n} = 120\cdot 5^{2n-1} > 5^{2n-1}+1.$$
		Also $q \geq k$ by \eqref{eq3.4d}, so that $q+l>k + 5^{2n-1}$, contradicting \eqref{eq3.4a}.
		
		\vskip 2mm
		\noindent
		\underline{Case 2.} $d_1 = d_2 = \cdots = d_l = b$, and $l >0$.
		
		In this case  the left-hand side of \eqref{eq3.4b} ends
		$a^{m_1} b^{5^{2j}+l}$, and this time we must have $5^{2j}+l = 5^{2n}$, with $j<n$. Then 
		$$l = 5^{2n} - 5^{2j} \geq 5^{2n} - 5^{2n-2} = \tfrac{26}{5} \cdot 5^{2n-1} >  5^{2n-1} + 1,$$
		so that, as in Case 1, $q+l > k + 5^{2n-1}$, a contradiction to \eqref{eq3.4a}.
		
		\vskip 2mm
		\noindent
		\underline{Case 3.} $\pdt{d} = 1$ (empty product).
		
		In this case  the left-hand side of \eqref{eq3.4b} ends $a^{m_1} b^{5^{2j}}$, and comparing with the right-hand side, we get $j=n$. Therefore Equation \eqref{eq3.4b} becomes
		$$x_1 x_2 \cdots x_q \, vb^{5^{2n-1}}v^{-1} a^{5^{2n}}b^{5^{2n}} = c^k a^{5^{2n}} b^{5^{2n}},$$
		which implies that 
		\begin{equation}	\label{eq3.4c}
			x_1 x_2 \cdots x_q = 
			c^k vb^{-5^{2n-1}} v^{-1}.
		\end{equation}		
		Let $\psi = \varphi_b - \varphi_c$, which is a homomorphism $\F_3 \to \Z$ that satisfies $\psi(x) \geq -1 \ (x \in X)$. Applying $\psi$ to both sides of \eqref{eq3.4c} we get
		\begin{equation*}	
			-k- 5^{2n-1} = \psi(x_1 x_2 \cdots x_q) 
			=  \sum_{i=1}^q \psi(x_i) \geq -q,
		\end{equation*}
		which implies that $q \geq k+ 5^{2n-1}$. In this scenario $l=0$, so that $q+l+1 \geq k + 5^{2n-1} +1$, which proves (i).
		
		To prove (ii) we reduce to (i).  Suppose that
		$d_1, \ldots, d_l \in S$, $x_1, \ldots x_q \in X$, and $j \in \N$ and $ v \in \F_3$ with $|v| \leq 5^j$, give
		\begin{equation} 	\label{eq3.4f}
			x_1 \cdots x_q \, x(v,j) \, \pdt{d} = c^{k_1} a^{5^{2n}} b^{5^{2n}}c^{k_2},
		\end{equation}
		with $q+l+1$ minimum possible.
		Then, as already observed, by Lemma \ref{3.4} the left-hand side of \eqref{eq3.4f} ends $a^{m_1}b^{m_2} d_{l'+1} \cdots d_l$, for some $1\leq l' \leq r$, and since $m_1>0$ the term $a^{m_1}$ must form part of $a^{5^{2n}}$. As such $c^{k_2}$ can only be represented within $d_{l'+1} \cdots d_l$, forcing $d_l= \cdots = d_{l-k_2+1} = c$. Write $d_i' = d_i \ (i = 1, \ldots, l-k_2)$. Then \eqref{eq3.4f} now becomes 
		$$x_1 \cdots x_q \, x(v,j) \, \pdt{d}' = c^{k_1} a^{5^{2n}} b^{5^{2n}}$$
		and by (i) $q + |\pdt{d}'| + 1 = q + l-k_2 + 1  = k_1 + 5^{2n-1} + 1$, so that $q + l +1 = k_1 + 5^{2n-1} + 1 + k_2.$ This proves (ii).
	\end{proof}
	
	Finally, we prove the second of our two main word length calculations. 
	
	\begin{proposition}		\label{3.8}
		Let $r \in \N$, and let $n_1, \ldots, n_r \geq 2$ and $k_1, \ldots, k_r$ be natural numbers satisfying 
		$$k_{i-1} > 3\cdot 5^{2n_i} \quad (i = 2,3, \ldots, r).$$
		Then
		\begin{equation}	\label{eq3.5}
			\left|a^{5^{2n_1}} b^{5^{2n_1}} c^{k_1} \cdots a^{5^{2n_r}}b^{5^{2n_r}} c^{k_r} \right|_X = \sum _{i = 1}^r k_i + \sum_{i=1}^r (5^{2n_i-1} +1).
		\end{equation}
	\end{proposition}
	
	\begin{proof}		
		It follows from Lemma \ref{3.5a} that the left-hand side of \eqref{eq3.5} is at most the right-hand side. We shall prove the reverse inequality.		
		We proceed by induction on $r$. The base case is given by Proposition \ref{3.5}(ii).
		
		Suppose that the equation
		\begin{equation} 	\label{eq3.6}
			\pdt{d}_1 x_1 \pdt{d}_2 \cdots \pdt{d}_q x_q \pdt{e} = a^{5^{2n_1}} b^{5^{2n_1}} c^{k_1} \cdots a^{5^{2n_r}}b^{5^{2n_r}} c^{k_r},
		\end{equation}
		where $d_{i,j} \in S$, $e_1, \ldots , e_l \in S$, and $x_1, \ldots, x_q \in X_\infty$, has the property that the left-hand expression is $X$-minimal. Without loss of generality we may assume that it is also super $X$-minimal. We shall prove that 
		$$q+l + \sum_{i=1}^q |\pdt{d}_i| \geq \sum _{i = 1}^r k_i + \sum_{i=1}^r (5^{2n_i-1} +1).$$
		
		By Lemma \ref{3.4} the product on the left-hand side of \eqref{eq3.6} ends 
		$$a^{m_1}b^{m_2} e_{l'} \cdots e_l$$
		as a reduced word, where 
		$m_1>0$, $m_2 \geq 0$, and $1 \leq l' \leq l$. 
		As such, comparing with the right-hand side of \eqref{eq3.6}, we must have $e_l = e_{l-1} = \cdots = e_{l-k_r+1} = c$. Letting $d_{q+1, i} = e_i \ (i =1 , \ldots, l-k_r)$, and cancelling the common factors of $c^{k_r}$ from both sides of \eqref{eq3.6}, we get
		\begin{equation}		\label{eq3.6a}
			\pdt{d}_1 x_1 \pdt{d}_2 \cdots \pdt{d}_q x_q \pdt{d}_{q+1} = a^{5^{2n_1}} b^{5^{2n_1}} c^{k_1} \cdots c^{k_{r-1}}a^{5^{2n_r}}b^{5^{2n_r}}.
		\end{equation}
		
		We fix for the remainder of the proof a sequence of cancellations that reduces the left-hand side of \eqref{eq3.6a}.  Under our fixed sequence of cancellations, we claim that at least one of the $c$'s from $c^{k_{r-1}}$ on the right-hand side of \eqref{eq3.6a} must have originated within one of the $\pdt{d}_i$'s on the left-hand side. In order to prove this we begin by making a brief analysis of the left-hand side.
		
		Reading from right to left, on the left-hand side of \eqref{eq3.6a} there must be a contiguous block of letters that reduces to $b^{5^{2n_r}}$, followed by a block that reduces to $a^{5^{2n_r}}$, followed by a block that reduces to $c^{k_{r-1}}$. We refer to these as the $b$-block, the $a$-block, and the $c$-block, respectively; they are not unique (as there may be subblocks at the boundary that completely cancel), but we fix a choice for the remainder of the proof. We make our choice such that the $c$-block is as short as possible. Together the three blocks form a single super-block that reduces to $c^{k_{r-1}} a^{5^{2n_r}} b^{5^{2n_r}}.$
		
		Assume towards a contradiction that none of the $c$'s in $c^{k_{r-1}}$ originate within a $\pdt{d}_i$. Then they must all originate within the $x_i$'s. Note that the endpoints of the $c$-block must be $c$'s by the `shortest' assumption, so (reading left to right) the $c$-block begins within say $x_Q$ and ends within  $x_{Q+R}$ (for  some $R \geq 0$).
		Write $x_{Q+i} = x(v_i, \tau_i) \ (i = 0, \ldots, R)$, and split up $x_Q$ as $x_Q = zy$ (reduced), where $y$ is part of the $c$-block, and $z$ is not. It follows that 
		\begin{equation} 	\label{eq3.22}
			x_Q \pdt{d}_{Q+1} \cdots \pdt{d}_q x_q \pdt{d}_{n+1} = z c^{k_{r-1}} a^{5^{2n_r}} b^{5^{2n_r}}.
		\end{equation}
		Moreover, since the $c$-block must begin with a letter contained in either $v_0$ or $v_0^{-1}$, we have $z|v_0 b^{5^{2\tau_0 -1 }} v_0^{-1}$, which implies that
		\begin{equation}	\label{eq3.23}
			|z| \leq 2 \cdot 5^{\tau_0} + 5^{2 \tau_0 - 1}.
		\end{equation}
		
		Let's say that we have $\gamma_i$ letters $c$ within $x_{Q+i} \ ( i = 0, \ldots, R)$ that remain uncancelled after the fixed sequence of cancellations is performed. By our assumption
		\begin{equation}   	\label{eq3.19}
			\gamma_0+\gamma_1 + \cdots + \gamma_R = k_{r-1}.
		\end{equation}
		For each $i = 0, \ldots, R$ we have
		$$\gamma_i \leq 2 |v_i| \leq 2 \cdot 5^{\tau_i},$$
		since $c$'s within an $x_{Q+i}$ must originate from $v_i$ or $v_i^{-1}$.
		It follows that 
		\begin{equation}		\label{eq3.20}
			2 \gamma_i < 5^{2\tau_i} \quad (i = 0, \ldots, R),
		\end{equation}
		and also that
		\begin{equation} 	\label{eq3.21}
			2 \gamma_0 + 2 \cdot (5^{2\tau_0 -1} + 2 \cdot 5^{\tau_0}) < 5^{2\tau_0}.
		\end{equation}
		
		We now apply $\varphi_a$ to both sides of \eqref{eq3.22}. We get 
		$$\varphi_a(x_Q \pdt{d}_{Q+1} \cdots x_n \pdt{d}_{q+1}) = \varphi_a(z) + 5^{2n_r}$$
		and expanding the left-hand side this becomes
		$$\sum_{i= Q+1}^{q+1} \varphi_a(\pdt{d}_i) + 5^{2 \tau_0} + \sum_{i=1}^R 5^{2\tau_i} + \sum_{i=R+1}^q \varphi_a(x_i)
		= \varphi_a(z) + 5^{2n_r},$$
		so that 
		\begin{align*} 
			\sum_{i= Q+1}^{q+1} \varphi_a(\pdt{d}_i) &= \varphi_a(z) + 5^{2n_r} -  5^{2 \tau_0} - \sum_{i=1}^R 5^{2\tau_i} - \sum_{i=R+1}^q \varphi_a(x_i) \\
			&\leq |z| + 5^{2n_r} -  5^{2 \tau_0} - \sum_{i=1}^R 5^{2\tau_i} \\
			&\leq |z| + 5^{2n_r} -  5^{2 \tau_0} -  2 \sum_{i=1}^R \gamma_i \qquad (\text{by \eqref{eq3.20}}) \\
			&\leq |z| + 5^{2n_r} - (2 \gamma_0 + 2 \cdot (5^{2\tau_0 -1} + 2 \cdot 5^{\tau_0} )) - 2 \sum_{i=1}^R \gamma_i \qquad (\text{by  \eqref{eq3.21}})\\
			&= (|z| - 5^{2\tau_0 -1} - 2 \cdot 5^{\tau_0}) +5^{2n_r} - (5^{2\tau_0 -1} + 2 \cdot 5^{\tau_0}) - 2k_{r-1} \qquad (\text{by \eqref{eq3.19}}) \\ 
			&\leq 5^{2n_r} - (5^{2\tau_0 -1} + 2 \cdot 5^{\tau_0}) - 2k_{r-1} \qquad (\text{by \eqref{eq3.23}}) \\
			&< -2 \cdot 5^{2n_r} - k_{r-1} - (5^{2\tau_0 -1} + 2 \cdot 5^{\tau_0}) \qquad (\text{because } k_{r-1}> 3\cdot 5^{2n_r}).
		\end{align*}
		Using Lemma \ref{3.0}, it follows that 
		\begin{align*}
			\sum_{i=Q+1}^{q+1} |\pdt{d}_k| &\geq \left| \sum_{i= Q+1}^{q+1} \varphi_a(\pdt{d}_i) \right| 
			> 2 \cdot 5^{2n_r} + k_{r-1} + (5^{2\tau_0 -1} + 2 \cdot 5^{\tau_0}) \\
			&\geq 2 \cdot 5^{2n_r} + k_{r-1} + |z| 
			\geq |z c^{k_{r-1}} a^{5^{2n_r}} b^{5^{2n_r}}|.   	
		\end{align*}
		However, this last inequality means that replacing $x_{Q} \pdt{d}_{Q+1} \cdots x_q \pdt{d}_{q+1}$ by 
		$z c^{k_{r-1}} a^{5^{2n_r}} b^{5^{2n_r}}$ written in letters from $S$ would reduce the number of symbols on the left-hand side of \eqref{eq3.6a} (whilst maintaining the equality to the right-hand side), contradicting \eqref{eq3.6a} being $X$-minimal. This proves our claim that at least one of the $c$'s from $c^{k_{r-1}}$ on the right-hand side of \eqref{eq3.6a} originates in one of the $\pdt{d}_i$'s on the left-hand side.
		
		As such there exists $Q \in \{1, \ldots, q \}$ and $j \in \{ 1, \ldots, |\pdt{d}_Q| \}$ such that $d_{Q,j}$ is a $c$ that remains uncancelled and forms part of $c^{k_{r-1}}$ on the right-hand side of \eqref{eq3.6a}. Let $p_Q$ denote the number of letters in the sequence $d_Q$ and write
		$$\pdt{d}_Q' = d_{Q,1}d_{Q,2} \cdots d_{Q, j-1} \qquad \text{and} \qquad \pdt{d}_Q'' = d_{Q,j}d_{Q,j+1} \cdots d_{Q, p_Q}.$$    
		We have 
		\begin{equation}	\label{eq3.24}
			\pdt{d}_Q'' x_Q  \cdots \pdt{d}_q x_q \pdt{d}_{q+1} = c^\nu a^{5^{2n_r}}b^{5^{2n_r}} 
		\end{equation}
		and
		\begin{equation} 	\label{eq3.25}
		\pdt{d}_1 x_1 \pdt{d}_2 \cdots \pdt{d}_{Q-1} x_{Q-1} \pdt{d}_Q' = a^{5^{2n_1}} b^{5^{2n_1}} c^{k_1} \cdots a^{5^{2n_{r-1}}}b^{5^{2n_{r-1}}} c^{k_{r-1} -\nu},
		\end{equation}
		for some $\nu \in \{ 1, \ldots, k_{r-1} \}$. 
		By Proposition \ref{3.5}(i) applied to \eqref{eq3.24}
		$$q+1-Q + |\pdt{d}_Q''| + \sum_{i=Q+1}^{q+1} |\pdt{d}_i| \geq \nu + 5^{2n_r - 1}$$
		and by the induction hypothesis applied to \eqref{eq3.25}
		$$Q-1 + \sum_{i=1}^{Q-1} |\pdt{d}_i| + |\pdt{d}_Q'| 
		\geq k_{r-1}-\nu + \sum _{i = 1}^{r-2} k_i +  \sum_{i=1}^{r-1} (5^{2n_i-1} +1).$$
		Adding these inequalities together we have
		$$q + \sum_{i=1}^{q+1} |\pdt{d}_i| \geq \sum _{i = 1}^{r-1} k_i + \sum_{i=1}^r (5^{2n_i-1} +1).$$
		We now return to the beginning of the proof 
		and recall that $l \ (=  |\pdt{e}|) \ = l-k_r + k_r = |\pdt{d}_{q+1}| + k_r.$ Therefore adding $k_r$ to both sides of the above inequality yields
		$$q +l + \sum_{i=1}^q |\pdt{d}_i| \geq \sum _{i = 1}^r k_i + \sum_{i=1}^r (5^{2n_i-1} +1),$$
		as required.
	\end{proof}

	\section{Proof of the Main Result}
	\noindent
	We define a weight $\omega$ on $\F_3$ by $\omega(t) = \exp(|t|_X) \ (t \in \F_3)$. This section is devoted to proving our main result of the paper, Theorem \ref{4.6}, which states that the two radicals of $\ell^1(\F_3, \omega)^{**}$ are different.	
	Given $u \in \F_3$ we set $\widetilde{\delta}_u = \frac{1}{\omega(u)} \delta_u$.
	
	\begin{lemma}		\label{4.1}
		Let $j \in \N$ and $u \in \F_3$. Then there exists $N = N(j,u) \in \N$ such that for all $k \geq N$
		$$\| \widetilde{\delta}_{a^{5^{2j}} b^{5^{2j}}} * \widetilde{\delta}_u * \widetilde{\delta}_{a^{5^{2k}} b^{5^{2k}}} \|_\omega 
		\leq \exp(-1 - 5^{2j-1}).$$
	\end{lemma}
	
	\begin{proof}
		Choose $N \in \N$ such that $5^N \geq |u|$ and $N > j+1$.
		For all $k \geq N$, we have
		\begin{align*}
			a^{5^{2j}}b^{5^{2j}} u a^{5^{2k}} b^{5^{2k}} &= 
			(a^{5^{2j}} b^{5^{2j} - 5^{2k-1}} u) \cdot
			(u^{-1} b^{5^{2k-1}} u a^{5^{2k}}b^{5^{2k}}) = \\
			&= (a^{5^{2j}} b^{5^{2j} - 5^{2k-1}} u) \cdot x(u^{-1}, k).
		\end{align*}
		By the choice of $N$ we have $x(u^{-1}, k) \in X$, and $5^{2j} - 5^{2k-1} <0$. As such we have
		\begin{align*}
			|a^{5^{2j}}b^{5^{2j}} u a^{5^{2k}} b^{5^{2k}} |_X 
			&\leq |a^{5^{2j}}b^{5^{2j} - 5^{2k-1}}|_X + |u|_X +|x(u^{-1}, k)|_X \\
			&\leq 5^{2j} + 5^{2k-1} - 5^{2j} + |u|_X + 1 = 5^{2k-1} + |u|_X +1.
		\end{align*}
		It follows, using \eqref{eqstar} in the third line, that 
		\begin{align*}
			\| \widetilde{\delta}_{a^{5^{2j}} b^{5^{2j}}} * \widetilde{\delta}_u * &\widetilde{\delta}_{a^{5^{2k}} b^{5^{2k}}} \|_\omega  = \frac{\omega(a^{5^{2j}}b^{5^{2j}} u a^{5^{2k}} b^{5^{2k}})}{\omega(a^{5^{2j}}b^{5^{2j}}) \omega(u) \omega(a^{5^{2k}} b^{5^{2k}})} \\
			&= \exp \left(| a^{5^{2j}}b^{5^{2j}} u a^{5^{2k}} b^{5^{2k}}|_X - 
			| a^{5^{2j}}b^{5^{2j}}|_X - |u|_X - |a^{5^{2k}} b^{5^{2k}}|_X \right) \\
			&\leq \exp \left( 5^{2k-1} + |u|_X + 1 - (5^{2j-1} + 1) 
			-|u|_X - ( 5^{2k-1}+1) \right) \\
			&= \exp(-1 - 5^{2j-1}),
		\end{align*}
		as claimed.
	\end{proof}
	
	
	
	\begin{lemma}		\label{4.3}
		Let $f \in \ell^1(\F_3, \omega)$ be finitely supported such that $\|f\|_\omega \leq 1$. Then there exist $N(j,f) \in \N$ such that 
		$$\| \widetilde{\delta}_{a^{5^{2j}} b^{5^{2j}}} * f * \widetilde{\delta}_{a^{5^{2k}} b^{5^{2k}}} \|_\omega < \exp(-1 -5^{2j-1}) \qquad (k \geq N(j,f)).$$
	\end{lemma}
	
	\begin{proof}
		Let  $N(j,u)$ be as in Lemma \ref{4.1}. Write $F = \supp f$, and let $N(j,f) = \max \{ N(j,u) : u \in F \}.$ Note that
		$$f = \sum_{u \in F} f(u) \delta_u = \sum_{u \in F} f(u) \omega(u) \widetilde{\delta}_u ,$$
		and that $\sum_{u \in F} |f(u)|\omega(u) \leq 1.$ Given $j \geq R$ and $k \geq N(j,f)$ we have
		\begin{align*}
			\| \widetilde{\delta}_{a^{5^{2j}} b^{5^{2j}}} * f * \widetilde{\delta}_{a^{5^{2k}} b^{5^{2k}}} \|_\omega 
			&= \left\| \sum_{u \in F} f(u) \omega(u) \widetilde{\delta}_{a^{5^{2j}}b^{5^{2j}}} * \widetilde{\delta}_u * \widetilde{\delta}_{a^{5^{2k}}b^{5^{2k}}} \right\|_\omega \\
			&\leq \sum_{u \in F} |f(u)|\omega(u) 
			\| \widetilde{\delta}_{a^{5^{2j}} b^{5^{2j}}} * \widetilde{\delta}_u * \widetilde{\delta}_{a^{5^{2k}} b^{5^{2k}}} \|_\omega \\
			&< \sum_{u \in F} |f(u)|\omega(u) \exp(-1 -5^{2j-1})  \leq \exp(-1 -5^{2j-1}) ,
		\end{align*}
		where we have used Lemma \ref{4.1} in the last line.
	\end{proof}
	
	Let $g_k = \widetilde{\delta}_{a^{5^{2k}}b^{5^{2k}}} \ (k \in \N).$ By the Banach--Alaoglu Theorem there exists a subnet
	$(g_{h(\beta)})_{\beta \in M}$ which weak*-converges to a limit in the unit ball of $\ell^1(\F_3, \omega)^{**}$. We denote this limit by $\Phi_0$. 
	Note that $\Phi_0 \neq 0$ because $$\langle \Phi_0, \omega \rangle = \lim_{\beta} \langle g_{h(\beta)}, \omega \rangle  = 1.$$
	We shall show that $\Phi_0 \in \rad( \ell^1(\F_3, \omega)^{**}, \Box)$.

	\begin{proposition}		\label{4.4}
		Let $\Psi \in \ell^1(\F_3, \omega)^{**}$. Then $\Phi_0 \Box \Psi \Box \Phi_0 = 0$.
	\end{proposition}
	
	\begin{proof}
		Without loss of generality we may take $\| \Psi \| \leq 1$. Let $(f_\alpha)_{\alpha \in \Lambda}$ be a net of finitely-supported functions in the unit ball of $\ell^1(\F_3, \omega)$ which converges to $\Psi$ in the weak*-topology on
		$ \ell^1(\F_3, \omega)^{**}$.
		We have
		$$\Phi_0 \Box \Psi \Box \Phi_0 = \lim_{w^*, \gamma} \, \lim_{w^*, \alpha} \, \lim_{w^*, \beta} \, g_{h(\gamma)}*f_\alpha*g_{h(\beta)}.$$
		Let $\eps >0$. Let $\lambda \in \ell^1(\F_3, \omega)^*$, and fix $\alpha$ and $\gamma$. Using the notation of Lemma \ref{4.3}, there exists $\beta_0 \in M$ such that for every $\beta \geq  \beta_0$ we have $h(\beta) \geq N(h(\gamma), f_\alpha).$ As such, Lemma \ref{4.3} implies that
		\begin{equation*}
			\left| \langle g_{h(\gamma)}*f_\alpha *g_{h(\beta)}, \lambda \rangle \right| \leq \| g_{h(\gamma)}*f_\alpha *g_{h(\beta)} \|_\omega \|\lambda \| 
			\leq \exp(-1- 5^{2h(\gamma) - 1}) \| \lambda \|,
		\end{equation*}
		and hence (taking the limits $\beta \to \infty$ and then $\alpha \to \infty$) that
		$$\left| \langle g_{h(\gamma)} \Box \Psi \Box \Phi_0, \lambda \rangle \right| \leq  \exp(-1- 5^{2h(\gamma) - 1}) \| \lambda \|.$$
		Let $\gamma_0 \in M$ satisfy $\exp(-1-5^{2h(\gamma) - 1}) < \eps$ for all $\gamma > \gamma_0$. Then
		$$\left| \langle g_{h(\gamma)} \Box \Psi \Box \Phi_0, \lambda \rangle \right| < \eps \| \lambda \| \quad (\gamma \geq \gamma_0),$$
		which implies that
		$$\left| \langle \Phi_0 \Box \Psi \Box \Phi_0, \lambda \rangle \right| \leq \eps \| \lambda \|,$$
		and hence (as $\lambda$ was arbitrary) that $\| \Phi_0 \Box \Psi \Box \Phi_0 \| \leq \eps$. As $\eps >0$ was arbitrary we have $\Phi_0 \Box \Psi \Box \Phi_0 = 0$, as required.
	\end{proof}
	
	It follows that $(\Psi \Box \Phi_0)^{\Box 2} = 0$ for every  $\Psi \in \ell^1(\F_3, \omega)^{**}$. Hence $\Phi_0$ generates a nilpotent left ideal in $\ell^1(\F_3, \omega)^{**}$, which implies that $\Phi_0 \in \rad( \ell^1(\F_3, \omega)^{**}, \Box)$.
	
	Let $\psi_j= \widetilde{\delta}_{c^j} = \delta_{c^j} \ (j \in \N)$. Let $(\psi_{m(\alpha)})_{\alpha \in \Lambda}$ be a weak*-convergent subnet of $(\psi_j)$, which exists by the Banach--Alaoglu Theorem. Define $\Psi_0 = \lim_{w^*,\alpha} \psi_{m(\alpha)}$. 
	
	\begin{proposition}		\label{4.5}
		The functional $\Phi_0 \Diamond \Psi_0$ is not quasi-nilpotent in $(\ell^1(\F_3, \omega)^{**}, \Diamond)$.
	\end{proposition}
	
	\begin{proof}
		Note that  $\omega \in \ell^\infty(\F_3, 1/\omega) = \ell^1(\F_3, \omega)^*$, and $\| \omega \| = 1$. Let $k \in \N$.  Then
		$$(\Phi_0 \Diamond \Psi_0)^{\Diamond k} = \lim_{\alpha_k} \lim_{\beta_k} \cdots \lim_{\alpha_1} \lim_{\beta_1} g_{h(\beta_1)}*\psi_{m(\alpha_1)}* \cdots 
		* g_{h(\beta_k)}*\psi_{m(\alpha_k)},$$
		where the limits are taken in the weak*-topology on $\ell^1(\F_3, \omega)^{**}$. 
		By \eqref{eqdoubledagger}, provided that $m(\alpha_{i-1})>3 \cdot 5^{h(\beta_i)} \ (i = 2, 3, \ldots, k)$, we have 
		\begin{multline*}
			\left| a^{5^{2h(\beta_1)}}b^{5^{2h(\beta_1)}} c^{m(\alpha_1)} \cdots 
			a^{5^{2h(\beta_k)}}b^{5^{2h(\beta_k)}} c^{m(\alpha_k)} \right|_X  
			- \left|  a^{5^{2h(\beta_1)}}b^{5^{2h(\beta_1)}} \right|_X - 
			\left| c^{m(\alpha_1)} \right|_X - 
			\cdots \\
			\cdots
			- \left| a^{5^{2h(\beta_k)}}b^{5^{2h(\beta_k)}} \right|_X 
			-\left| c^{m(\alpha_k)}  \right|_X = 0,
		\end{multline*}
		and hence
		\begin{align*}
			\langle  g_{h(\beta_1)}*&\psi_{m(\alpha_1)}* \cdots 
			* g_{h(\beta_k)}*\psi_{m(\alpha_k)}, \omega \rangle  \\
			&= \frac{\omega \left[ a^{5^{2h(\beta_1)}}b^{5^{2h(\beta_1)}} c^{m(\alpha_1)} \cdots 
				a^{5^{2h(\beta_k)}}b^{5^{2h(\beta_k)}} c^{m(\alpha_k)} \right]}
			{\omega(a^{5^{2h(\beta_1)}}b^{5^{2h(\beta_1)}}) \omega(c^{m(\alpha_1)}) \cdots 
				\omega(a^{5^{2h(\beta_k)}}b^{5^{2h(\beta_k)}}) \omega(c^{m(\alpha_k)})} \\
			&=\exp(0) = 1.
		\end{align*}
		As such
		\begin{equation*}
			\langle (\Phi_0 \Diamond \Psi_0)^{\Diamond k} , \omega \rangle = 
			\lim_{\alpha_k} \lim_{\beta_k} \cdots \lim_{\alpha_1} \lim_{\beta_1} \langle  g_{h(\beta_1)}*\psi_{m(\alpha_1)}* \cdots 
			* g_{h(\beta_k)}*\psi_{m(\alpha_k)}, \omega \rangle = 1. 
		\end{equation*}
		Since $\| \omega \| = 1$ as a functional, this implies that $\|(\Phi_0 \Diamond \Psi_0)^{\Diamond k} \| \geq 1 \ (k \in \N)$, and hence also that  $\|(\Phi_0 \Diamond \Psi_0)^{\Diamond k} \|^{1/k} \geq 1 \ (k \in \N)$. Therefore (by the spectral radius formula) $\Phi_0 \Diamond \Psi_0$ is not quasi-nilpotent.
	\end{proof}
	
	Putting the last two propositions together gives our main theorem.
	
	\begin{theorem}		\label{4.6}
		We have $\Phi_0 \in \rad(\ell^1(\F_3, \omega)^{**}, \Box)$ but $\Phi_0 \notin \rad(\ell^1(\F_3, \omega)^{**}, \Diamond).$
	\end{theorem}
	
	\begin{proof}
		By Proposition \ref{4.4} $\Phi_0$ generates a nilpotent left ideal with respect to $\Box$, and hence $\Phi_0 \in \rad(\ell^1(\F_3, \omega)^{**}, \Box)$. However, by Proposition \ref{4.5}, the right ideal generated by $\Phi_0$ with respect to $\Diamond$ is not contained inside the set of quasi-nilpotent elements, and hence $\Phi_0 \notin \rad(\ell^1(\F_3, \omega)^{**}, \Diamond)$.
	\end{proof}
	
	\vskip 2mm
	\begin{remark} 	\label{DDRemark}
	The radical element $\Phi_0$ of $(\ell^1(\F_3, \omega)^{**}, \Box)$ that we have defined is a limit of normalised point-masses. Such elements were studied by Dales and Dedania in \cite{DD}, and in their notation (see Section 2.4 of this article) we have found an element of $(\beta \F_3)_\omega$ which belongs to one radical but not to the other. 
	However, note that in the unweighted case this would not be possible as such elements can never be radical: that is to say, if $G$ is a group and $(\delta_{t_\alpha})$ is a net of point masses in $\ell^1(G)$, weak*-converging to some element $\Gamma \in \ell^1(G)^{**}$, then $\Gamma \notin \rad(\ell^1(G)^{**}, \Box) \cup \rad(\ell^1(G)^{**}, \Diamond)$.
	The proof of this fact is along very similar lines to the proof of Proposition \ref{4.5} above: one shows that $\langle \Gamma^{\Box k}, 1 \rangle = \langle \Gamma^{\Diamond k}, 1 \rangle  = 1$ for all $k\in \N$, and then uses the spectral radius formula to conclude that $\Gamma$ is not quasi-nilpotent for either product.
\end{remark}

	\subsection*{Acknowledgments}
	This work was begun after the author had the opportunity to speak on the topic of Arens products at an NBFAS meeting in honour of the memory of H. Garth Dales (1944--2022), which took place in April 2024. The author is grateful to the organisers of that meeting for the opportunity to speak and for their hospitality.

\end{document}